\documentclass[a4paper,10pt]{amsart}

\allowdisplaybreaks[1]

\usepackage{amsmath}
\usepackage{amsthm}
\usepackage{amssymb}
\usepackage{mathrsfs}
\usepackage[all]{xy}

\newcommand{\QQ}{\mathbb{Q}}
\newcommand{\RR}{\mathbb{R}}
\newcommand{\CC}{\mathbb{C}}
\newcommand{\HH}{\operatorname{H}}
\newcommand{\rH}{\widetilde{\operatorname{H}}}
\newcommand{\tensor}{\otimes}
\newcommand{\LL}{\mathbb{L}}
\newcommand{\Com}{Com}
\newcommand{\Lie}{Lie}
\newcommand{\Coop}{\mathscr{C}}
\newcommand{\Op}{\mathscr{A}}
\newcommand{\adjunction}[4]{\xymatrix{ #1 \ar@<1ex>[rr]^-{#3} && #2 \ar@<1ex>[ll]^-{#4}}}
\newcommand{\G}{\mathbb{G}}
\newcommand{\Hom}{\operatorname{Hom}}

\newcommand{\Tor}{\operatorname{Tor}}
\newcommand{\as}{\text{\normalfont{<}}}
\newcommand{\Diag}{\mathscr{D}}
\newcommand{\im}{\operatorname{im}}
\newcommand{\Tw}{\operatorname{Tw}}
\newcommand{\Kos}{\operatorname{Kos}}

\newtheorem{theorem}{Theorem}
\newtheorem{proposition}[theorem]{Proposition}
\newtheorem{corollary}[theorem]{Corollary}

\theoremstyle{definition}
\newtheorem{definition}[theorem]{Definition}
\newtheorem{example}[theorem]{Example}
\newtheorem{remark}[theorem]{Remark}

\title{Koszul spaces}
\author{Alexander Berglund}
\address{Department of Mathematical Sciences, University of Copenhagen}
\email{alexb@math.ku.dk}
\thanks{Supported by the Danish National Research Foundation (DNRF) through the Centre for Symmetry and Deformation}

\begin{document}

\begin{abstract}
We prove that a nilpotent space is both formal and coformal if and only if it is rationally homotopy equivalent to the derived spatial realization of a graded commutative Koszul algebra. We call such spaces \emph{Koszul spaces} and we show that the rational homotopy groups and the rational homology of iterated loop spaces of Koszul spaces can be computed by applying certain Koszul duality constructions to the cohomology algebra.
\end{abstract}

\maketitle

\section{Introduction}
The purpose of this paper is to develop some ideas from Koszul duality theory within the framework of rational homotopy theory. The technology of Koszul duality for algebras over Koszul operads does not seem to have been exploited to its full potential in algebraic topology yet --- at least not explicitly --- and this paper is meant to take steps in this direction.

Our main result is that a space is simultaneously formal and coformal if and only if it is a \emph{Koszul space}, by which we mean that it is rationally homotopy equivalent to the derived spatial realization of a graded commutative Koszul algebra. Recall that a space is called formal if its rational homotopy type is a formal consequence of its rational cohomology algebra \cite[\S 12]{Sullivan}. This means that it is in principle possible to extract any rational homotopy invariant, such as the rational homotopy groups, from the cohomology algebra. However, doing this in practice entails the non-trivial algebraic problem of constructing a minimal model for the cohomology. Similar remarks apply to coformal spaces, which are spaces whose rational homotopy type is a formal consequence of the rational homotopy Lie algebra. A consequence of our main result is that for spaces that are \emph{both} formal and coformal, the rational homotopy groups, with Whitehead products, and the rational homology of any iterated loop space, with Pontryagin products and Browder brackets, can be extracted \emph{directly} from the cohomology by applying certain Koszul duality constructions. A key technical result is Theorem \ref{thm:koszul algebra} which should be of independent interest; it has as a striking consequence an intrinsic characterization of Koszul algebras (over a Koszul operad) in terms of formality of the bar construction. The validity of such a characterization seems to have been generally accepted as true by experts, but we have been unable to find an account of this in the literature.

The point of view presented in this paper will serve as a background for the theory of \emph{Koszul models} that is the subject of a sequel paper, and of which the theory developed in this paper is a facet.

\vskip10pt

Before we can give the precise statements of our main results, Theorems 2, 3 and 4, we need to recall some facts from rational homotopy theory and Koszul duality theory. Recall \cite{BG,FHT-RHT,Quillen,Sullivan} that the rational homotopy type of a connected nilpotent space $X$ of finite $\QQ$-type is modeled algebraically by a commutative differential graded algebra $A_{PL}^*(X)$ with cohomology $\HH^*(X;\QQ)$, or alternatively by a differential graded Lie algebra $\lambda(X)$ with homology the rational homotopy groups $\pi_*(\Omega X)\tensor \QQ$ with Samelson products (the degree $0$ part has to be given an appropriate interpretation if $\pi_1 (X)$ is non-abelian, see Theorem \ref{thm:nil}).
\begin{itemize}
\item $X$ is called \emph{formal} if $A_{PL}^*(X)$ is quasi-isomorphic to $\HH^*(X;\QQ)$ as a commutative differential graded algebra.
\item $X$ is called \emph{coformal} if $\lambda(X)$ is quasi-isomorphic to $\pi_*(\Omega X)\tensor \QQ$ as a differential graded Lie algebra.
\end{itemize}
We refer to \cite{NM} for a further discussion about formality and coformality. The contravariant functor $A_{PL}^*$ from spaces to commutative differential graded algebras induces an equivalence between the homotopy categories of connected nilpotent rational spaces of finite $\QQ$-type and minimal algebras of finite type. The inverse is given by the \emph{spatial realization} functor $\langle - \rangle$, see \cite{BG,FHT-RHT,Sullivan}. By the \emph{derived spatial realization} of a commutative differential graded algebra we mean the spatial realization of its minimal model. Finally, recall that a graded commutative algebra $A$ is called a \emph{Koszul algebra} if it is generated by elements $x_i$ modulo certain quadratic relations
$$\sum_{i,j} c_{ij} x_ix_j = 0$$
such that $\Tor_{i,j}^A(\QQ,\QQ) = 0$ for $i\ne j$ \cite{Priddy}. Here, the extra grading on $\Tor$ comes from the weight grading on $A$ determined by assigning weight $1$ to the generators $x_i$. There is also a notion of a \emph{Koszul Lie algebra}, see Section \ref{sec:koszul}.

\begin{definition}
A \emph{Koszul space} is a space which is rationally homotopy equivalent to the derived spatial realization of a graded commutative Koszul algebra.
\end{definition}
It is not difficult to see that a Koszul space is both formal and coformal. Our first main result implies that the converse is true.
\begin{theorem} \label{thm:koszul space}
Let $X$ be a connected nilpotent space of finite $\QQ$-type. The following are equivalent:
\begin{enumerate}
\item $X$ is both formal and coformal. \label{ks1}
\item $X$ is formal and the cohomology $\HH^*(X;\QQ)$ is a Koszul algebra. \label{ks2}
\item $X$ is coformal and the rational homotopy $\pi_*(\Omega X)\tensor \QQ$ is a Koszul Lie algebra. \label{ks3}
\item $X$ is a Koszul space. \label{ks4}
\end{enumerate}
\end{theorem}
The proof is given at the end of Section \ref{sec:rht} and depends on results from Sections \ref{sec:koszul} and \ref{sec:rht}. Examples of Koszul spaces include configuration spaces of points in $\RR^n$, highly connected manifolds, suspensions, loop spaces, wedges and products of Koszul spaces. These examples are described in more detail in the last section of the paper, Section \ref{sec:examples}.

If $A$ is a graded commutative Koszul algebra its \emph{Koszul dual Lie algebra} $A^{!_{\Lie}}$ is the free graded Lie algebra on generators $\alpha_i$, of homological degree $|\alpha_i| = |x_i|-1$, modulo the \emph{orthogonal relations}. This means that a relation
$$\sum_{i,j} \lambda_{ij} [\alpha_i,\alpha_j] = 0$$
holds if and only if
$$\sum_{i,j} (-1)^{|x_i||\alpha_j|} c_{ij}\lambda_{ij} = 0$$
whenever the coefficients $c_{ij}$ represent a relation among the generators $x_i$ for $A$.

\begin{theorem} \label{thm:koszul dual}
If $X$ is a Koszul space, then homotopy and cohomology are Koszul dual in the sense that there is an isomorphism of graded Lie algebras
$$\pi_*(\Omega X)\tensor \QQ = \HH^*(X;\QQ)^{!_{\Lie}}.$$
\end{theorem}
Theorem \ref{thm:koszul dual} is proved along with Theorem \ref{thm:koszul space} at the end of Section \ref{sec:rht}. Koszul duality between commutative algebras and Lie algebras is an instance of Koszul duality for algebras over Koszul operads; the commutative operad $\Com$ is a Koszul operad and its Koszul dual operad is the Lie operad $\Lie = \Com^!$ \cite{GK}. If $A$ is a Koszul algebra over a Koszul operad $\Op$ then it has a Koszul dual $A^!$ which is an algebra over the Koszul dual operad $\Op^!$. Any Koszul $\Op$-algebra $A$ admits a quadratic presentation and the Koszul dual $A^!$ can be computed by taking an orthogonal quadratic presentation, see Theorem \ref{thm:quadratic}. Another important Koszul operad is the homology of the little $n$-cubes operad $\G_n = \HH_*(E_n)$ \cite{BV,fcohen,May}. The operad $\G_n$ is Koszul self-dual up to a suspension, $\G_n^! = \Sigma^{n-1} \G_n$ \cite{GJ}, so a Koszul $\G_n$-algebra $A$ has a Koszul dual $\G_n$-algebra $A^{!_{\G_n}}$. The little $n$-cubes operad acts on $n$-fold loop spaces $\Omega^n X$ and so $\G_n$ acts on the homology $\HH_*(\Omega^n X;\QQ)$. Furthermore, any commutative algebra, such as $\HH^*(X;\QQ)$, can be viewed as a $\G_n$-algebra in a natural way. In Section \ref{sec:g} we prove the following result.

\begin{theorem} \label{thm:loop}
If $X$ is an $n$-connected Koszul space, then loop space homology and cohomology are Koszul dual in the sense that there is an isomorphism of $\G_n$-algebras
$$\HH_*(\Omega^n X;\QQ) = \HH^*(X;\QQ)^{!_{\G_n}}.$$
\end{theorem}

Before getting into the proofs of Theorems 2, 3 and 4, we list some simple but interesting consequences of these results. First off, Theorem \ref{thm:koszul space} settles in the positive a question of Papadima and Suciu \cite[Question 8.2]{PS}.
\begin{corollary} \label{cor:arrangements}
Let $\mathscr{A}$ be an arrangement of hyperplanes in $\CC^\ell$ and let $X=M(\mathscr{A})$ denote its complement. Let $\mathscr{A}^{k+1}$ denote the corresponding redundant subspace arrangement in $\CC^{(k+1)\ell}$ and let $Y=M(\mathscr{A}^{k+1})$. Then $Y$ is coformal if and only if $\HH^*(X;\QQ)$ is a Koszul algebra.
\end{corollary}

\begin{proof}
Both $X$ and $Y$ are formal, see \cite[\S 1.5]{PS}, so if $Y$ is also coformal then it follows from Theorem \ref{thm:koszul space} that $\HH^*(Y;\QQ)$ is a Koszul algebra. From the rescaling formula
$$\HH^*(Y;\QQ) = \HH^*(X;\QQ)[k],$$
see \cite[\S 1.3]{PS}, it easily follows that $\HH^*(X;\QQ)$ is Koszul as well.
\end{proof}

In 1965, Serre \cite[IV-52]{Serre} asked whether the Poincar\'e series
$$\sum_{i\geq 0} \dim_\QQ \HH_i(\Omega X;\QQ) z^i$$
of a simply connected finite CW-complex $X$ is always a rational function. This question remained open until a counterexample was constructed by Anick in 1982 \cite{Anick1,Anick2}. It is interesting to note that such counterexamples cannot be Koszul spaces.

\begin{corollary}
If $X$ is a simply connected Koszul space with finitely generated rational cohomology algebra then $\HH_*(\Omega X;\QQ)$ has rational Poincar\'e series.
\end{corollary}

\begin{proof}
For a bigraded algebra $A$ let
$$A(t,z) = \sum_{i,j} \dim_\QQ A_{i}(j) z^i t^j$$
where $t$ keeps track of the weight grading and $z$ keeps track of the homological grading. The bigraded Poincar\'e series of a Koszul algebra $A$ and its Koszul dual associative algebra $A^!$ are rationally related by the following formula:
$$A^!(t,z) = A(-tz^{-1},z)^{-1}.$$
In particular, $A^!(1,z)$ is rational if $A(t,z)$ is rational. It is well known that any finitely generated graded commutative algebra has rational Poincar\'e series, so the claim follows from Theorem \ref{thm:loop} with $n=1$; $\HH_*(\Omega X;\QQ) = ·\HH^*(X;\QQ)^!$.
\end{proof}

\begin{corollary}
If $X$ is a Koszul space with finitely generated cohomology algebra, then the rational homotopy Lie algebra $\pi_*(\Omega X)\tensor \QQ$ is finitely generated.
\end{corollary}

\begin{proof}
The Koszul dual Lie algebra of a finitely generated commutative Koszul algebra is finitely generated, so the claim follows from Theorem \ref{thm:koszul dual}.
\end{proof}

Our results also shed new light on existing results. The following corollary recovers a well known result due to F.Cohen \cite{fcohen}, but our short proof using Koszul duality is new. We like to think of Theorem \ref{thm:loop} as a generalization of this result.

\begin{corollary}
For any space $Y$ there is an isomorphism of $\G_n$-algebras
$$\HH_*(\Omega^n \Sigma^n Y;\QQ) \cong \G_n[\rH_*(Y;\QQ)]$$
where the right hand side denotes the free $\G_n$-algebra on the reduced homology of $Y$.
\end{corollary}

\begin{proof}
The $n$-fold suspension $\Sigma^n Y$ of any space $Y$ is an $n$-connected Koszul space whose cohomology $\HH^*(\Sigma^n Y;\QQ)$ is a trivial algebra `generated' by $\rH{}^*(\Sigma^n Y;\QQ) \cong s^{-n}\rH{}^*(Y;\QQ)$. The Koszul dual $\G_n$-algebra of a trivial $\G_n$-algebra generated by a graded vector space $V$ is the free $\G_n$-algebra $\G_n[(s^n V)^\vee]$.
\end{proof}
In view of Cohen's calculation of $\HH_*(\Omega^n\Sigma^n Y;\mathbb{F}_p)$ \cite{fcohen}, it would be interesting to see to what extent Theorem \ref{thm:loop} extends to finite coefficients where one would also need to take into consideration Araki-Kudo-Dyer-Lashof and Steenrod operations.

Relations between the Koszul property, formality and coformality have been studied in \cite{PS} and \cite{PY}, but with a more restrictive definition of Koszul algebras; their definition require Koszul algebras to be generated in cohomological degree $1$. In \cite[Example 4.10]{PS} the authors give an example of a space which is formal and coformal, but where the authors say the cohomology algebra is not Koszul, namely $S^1\vee S^2$. The cohomology algebra is not Koszul according to the more restrictive definition, simply because it is not generated in cohomological degree $1$, but it is Koszul according to our definition. Theorem \ref{thm:koszul space} and Theorem \ref{thm:koszul dual} together imply the main result of \cite{PY} in the nilpotent case.

\begin{corollary}
Let $X$ be a formal nilpotent space. Then $X$ is a rational $K(\pi,1)$-space if and only if $\HH^*(X;\QQ)$ is Koszul and generated in cohomological degree $1$.
\end{corollary}

\begin{proof}
If $X$ is formal and $\HH^*(X;\QQ)$ is Koszul, then $X$ is a Koszul space and so by Theorem \ref{thm:koszul dual} we have that $\pi_*(\Omega X)\tensor \QQ \cong \HH^*(X;\QQ)^{!_{\Lie}}$. If $\HH^*(X;\QQ)$ is generated in cohomological degree $1$ then $\HH^*(X;\QQ)^{!_{\Lie}}$ is generated in homological degree $0$, which forces it to be concentrated in degree $0$. Hence $\pi_n(X)\tensor \QQ = 0$ for $n>1$.

Conversely, for any differential graded Lie algebra $L$ concentrated in non-negative homological degrees, there is a morphism $L\rightarrow \HH_0(L)$ which is a quasi-isomorphism if and only if $\HH_n(L) = 0$ for $n>0$. In particular, $L$ is coformal if it has homology concentrated in degree $0$. It follows that every rational $K(\pi,1)$-space is coformal. Thus, if $X$ is simultaneously formal and a rational $K(\pi,1)$, it is both formal and coformal and hence the cohomology is Koszul by Theorem \ref{thm:koszul space}.
\end{proof}

At the expense of getting less transparent statements, it should be possible to extend our main results to the not necessarily nilpotent case so as to include \cite[Theorem 5.1]{PY} as a corollary. The main difficulty is the lack of a Lie model $\lambda(X)$ for non-nilpotent spaces $X$. Also, one would have to invent an appropriate notion of coformality --- this notion does not seem to have been considered yet for non-nilpotent spaces.

\subsection*{Acknowledgments}
The author is grateful to Alexandru Suciu for directing his attention to the paper \cite{PY}. We also thank Bruno Vallette and Joan Mill\`es for discussions about Koszul duality for algebras over operads.

\section{Koszul duality for algebras over operads} \label{sec:koszul}
Koszul algebras were introduced by Priddy \cite{Priddy}. Koszul duality for operads was developed by Ginzburg-Kapranov \cite{GK}, Getzler-Jones \cite{GJ}. A modern comprehensive introduction to Koszul duality theory for operads can be found in the recent book \cite{LV} by Loday and Vallette. See also Fresse \cite{Fresse}. Koszul duality for algebras over Koszul operads was introduced in \cite{GK} and has recently been developed further by Mill\`es \cite{Milles}. In what follows, we will review the parts of the theory needed for proving the main results of this section --- Theorem \ref{thm:koszul}, Theorem \ref{thm:koszul algebra} and Corollary \ref{cor:formal} --- and we will freely use terminology from \cite{LV} without further reference.

\subsection{Twisting morphisms, bar and cobar constructions}
Let $\Coop$ be a coaugmented cooperad and $\Op$ an augmented operad in chain complexes and let $\tau\colon \Coop\rightarrow \Op$ be a twisting morphism. Let $C$ be a $\Coop$-coalgebra and let $A$ be an $\Op$-algebra. A \emph{twisting morphism relative to $\tau$} is a map $\kappa\colon C\rightarrow A$ of degree $0$ such that
$$\partial(\kappa) + \tau * \kappa = 0$$
where $\tau * \kappa$ is the composite map
$$\xymatrix{C\ar[r]^-\Delta & \Coop[C] \ar[r]^-{\tau[\kappa]} & \Op[A] \ar[r]^-\mu & A.}$$
Let $\Tw_\tau(C;A)$ denote the set of twisting morphisms relative to $\tau$. It is a bifunctor $\Coop_{coalg}^{op}\times \Op_{alg} \rightarrow Set$. For a fixed $\Op$-algebra $A$, the functor $\Tw_\tau(-;A)$ is represented by the \emph{bar construction} $B_\tau A$ and for a fixed $\Coop$-coalgebra $C$, the functor $\Tw_\tau(C;-)$ is represented by the \emph{cobar construction} $\Omega_\tau C$; there are natural bijections
\begin{equation} \label{bij}
\begin{array}{ccccc}
\Hom_\Op(\Omega_\tau C,A) & \cong & \Tw_\tau(C;A) & \cong & \Hom_\Coop(C,B_\tau A) \\
\psi_\kappa & & \kappa & & \phi_\kappa
\end{array}
\end{equation}
In particular, the bar and cobar constructions form an adjoint pair of functors
$$\adjunction{\Coop_{coalg}}{\Op_{alg}}{\Omega_\tau}{B_\tau}.$$
The bar construction is defined as the $\Coop$-coalgebra $B_\tau A = (\Coop[A],d+b)$ where $d$ is the internal differential of the chain complex $\Coop[A]$ and $b$ is the unique $\Coop$-coderivation making the diagram
$$\xymatrix{\Coop[A] \ar[r]^b \ar[d]^-{\tau[A]} & \Coop[A] \ar[d]^-{\epsilon[A]} \\ \Op[A] \ar[r]^-\mu & A}$$
commute. The fact that $(d+b)^2= 0$ is a consequence of the Maurer-Cartan equation for $\tau$.

The cobar construction is defined as the $\Op$-algebra $\Omega_\tau C = (\Op[C],d+\delta)$, where $d$ is the internal differential of $\Op[C]$ and $\delta$ is the unique $\Op$-derivation making the diagram
$$\xymatrix{C \ar[d]^-{\eta[C]} \ar[r]^-{\Delta} & \Coop[C] \ar[d]^-{\tau[C]} \\ \Op[C] \ar[r]^-\delta & \Op[C]}$$
commute. The fact that $(d+\delta)^2 = 0$ is a consequence of the Maurer-Cartan equation for $\tau$.

The bijection \eqref{bij} can be made explicit as follows. The unit and counit maps $\eta[C]\colon C\rightarrow \Op[C]$ and $\epsilon[A]\colon \Coop[A]\rightarrow A$ define twisting morphisms $\iota\colon C\rightarrow \Omega_\tau C$ and $\pi\colon B_\tau A \rightarrow A$ called the \emph{universal twisting morphisms}. Given a morphism of $\Coop$-coalgebras $\phi\colon C\rightarrow B_\tau A$ we get a twisting morphism $\kappa = \pi^*(\phi)\colon C\rightarrow A$. Similarly, given a morphism of $\Op$-algebras $\psi\colon \Omega_\tau C\rightarrow A$, we get a twisting morphism $\kappa = \iota_*(\phi)\colon C\rightarrow A$.
\begin{equation} \label{diag:universal}
\xymatrix{& B_\tau A \ar[dr]^-\pi \\ C \ar[rr]^-\kappa \ar[ur]^-{\phi_\kappa} \ar[dr]_-\iota && A  \\ & \Omega_\tau C \ar[ur]_-{\psi_\kappa}}
\end{equation}
Conversely, given a twisting morphism $\kappa\colon C\rightarrow A$, there is an associated morphism of $\Coop$-algebras $\phi_\kappa\colon C\rightarrow B_\tau A$ given by the composite map
$$\xymatrix{C\ar[r]^-\Delta & \Coop[C] \ar[r]^-{\Coop[\kappa]} & \Coop[A]}$$
and there is an associated morphism of $\Op$-algebras $\psi_\kappa\colon \Omega_\tau C\rightarrow A$ given by the composite map
$$\xymatrix{\Op[C] \ar[r]^-{\Op[\kappa]} & \Op[A] \ar[r]^-\mu & A.}$$
The unit and counit of the adjunction give natural morphisms of $\Op$-algebras and $\Coop$-coalgebras
$$\Omega_\tau B_\tau A \rightarrow A, \quad C\rightarrow  B_\tau \Omega_\tau C.$$
The twisting morphism $\tau\colon\Coop\rightarrow \Op$ is a \emph{Koszul twisting morphism} if these are quasi-isomorphisms for all $C$ and $A$.
\begin{definition}
A twisting morphism $\kappa\colon C\rightarrow A$ relative to $\tau$ is a \emph{Koszul twisting morphism} if the associated morphisms $\phi_\kappa\colon C\rightarrow B_\tau A$ and $\psi_\kappa \colon \Omega_\tau C \rightarrow A$ are quasi-isomorphisms. We will denote the set of Koszul twisting morphisms relative to $\tau$ from $C$ to $A$ by $\Kos_\tau(C;A)$.
\end{definition}
In particular, $\tau\colon \Coop\rightarrow \Op$ is a Koszul twisting morphism if and only if the universal twisting morphisms $\pi\colon B_\tau A\rightarrow A$ and $\iota\colon C\rightarrow \Omega_\tau C$ are Koszul for all $C$ and $A$.

\begin{definition}
\begin{itemize}
\item A morphism of $\Op$-algebras $\psi\colon A\rightarrow A'$ is called a \emph{weak equivalence} if it is a quasi-isomorphism.

\item A morphism of $\Coop$-coalgebras $\phi\colon C\rightarrow C'$ is called a \emph{weak equivalence} if the associated morphism of $\Op$-algebras $\Omega_\tau \phi\colon \Omega_\tau C\rightarrow \Omega_\tau C'$ is a quasi-isomorphism.
\end{itemize}
\end{definition}
Every weak equivalence of $\Coop$-coalgebras is a quasi-isomorphism but not conversely. A counterexample is given by the map of associative dg-coalgebras  $Bk[x]\rightarrow Bk[x,x^{-1}]$ induced by the algebra inclusion $k[x]\subseteq k[x,x^{-1}]$. The bar construction $B_\tau$ preserves but does not reflect quasi-isomorphisms. The cobar construction $\Omega_\tau$ reflects but does not preserve quasi-isomorphisms. However, the definition of weak equivalence is rigged so that the bar and cobar constructions both preserve and reflect weak equivalences. There are model category structures on $\Op$-algebras and $\Coop$-coalgebras with these weak equivalences, see \cite{Vallette}.

\begin{definition}
\begin{enumerate}
\item We say that two $\Op$-algebras $A$ and $A'$ are \emph{weakly equivalent} and write $A\sim A'$ if there is a zig-zag of weak equivalences of $\Op$-algebras
$$\xymatrix{A & \ar[l]_-\sim A_1 \ar[r]^-\sim & A_2 & \ar[l]_-\sim \ldots \ar[r]^-\sim & A_n & \ar[l]_-\sim A'}$$
We say that two $\Coop$-coalgebras $C$ and $C'$ are \emph{weakly equivalent} and write $C\sim C'$ if they can be connected through a zig-zag of weak equivalences of $\Coop$-coalgebras.

\item An $\Op$-algebra $A$ is called \emph{formal} if it is weakly equivalent to its homology, $A\sim \HH_*(A)$. A $\Coop$-coalgebra $C$ is called \emph{formal} if it is weakly equivalent to its homology $C\sim H_*(C)$.
\end{enumerate}
\end{definition}

\begin{proposition}
Let $\kappa\colon C\rightarrow A$ be a twisting morphism relative to a Koszul twisting morphism $\tau\colon \Coop\rightarrow \Op$. The following are equivalent
\begin{enumerate}
\item $\kappa$ is a Koszul twisting morphism.
\item $\psi_\kappa\colon \Omega_\tau C\rightarrow A$ is a weak equivalence.
\item $\phi_\kappa\colon C\rightarrow B_\tau A$ is a weak equivalence.
\end{enumerate}
\end{proposition}

\begin{proof}
Since $\tau$ is a Koszul twisting morphism, the natural morphism of $\Op$-algebras $\epsilon_\tau\colon\Omega_\tau B_\tau A \rightarrow A$ is a quasi-isomorphism. The composite morphism
$$\xymatrix{\Omega_\tau C \ar[r]^-{\Omega_\tau \phi_\kappa} & \Omega_\tau B_\tau A \ar[r]^-{\epsilon_\tau} & A}$$
is equal to $\psi_\kappa$. Therefore, $\psi_\kappa$ is a quasi-isomorphism if and only if $\Omega_\tau \phi_\kappa$ is.
\end{proof}

An \emph{$\infty$-morphism} $A\rightsquigarrow A'$ between $\Op$-algebras is a morphism of $\Coop$-coalgebras $B_\tau A\rightarrow B_\tau A'$, and an $\infty$-morphism $C\rightsquigarrow C'$ between $\Coop$-coalgebras is a morphism of $\Op$-algebras $\Omega_\tau C\rightarrow \Omega_\tau C'$. According to \cite[Theorem 11.4.14]{LV}, if $\Op$ is a Koszul operad then $A\sim A'$ if and only if there exists an $\infty$-quasi-isomorphism $A\rightsquigarrow A'$. From the bijections \eqref{diag:universal} it is clear that $\Tw_\tau(C,A)$ is functorial with respect to $\infty$-morphisms. Furthermore, $\Kos_\tau(C,A)$ is functorial with respect to $\infty$-quasi-isomorphisms. Since any $\infty$-quasi-isomorphism admits an inverse \cite[Theorem 10.4.7]{LV} we get in particular the following result.

\begin{proposition} \label{prop:transfer}
Let $\kappa\colon C\rightarrow A$ be a Koszul twisting morphism relative to a Koszul twisting morphism $\tau\colon \Coop \rightarrow \Op$ where $\Op$ is a Koszul operad. If $C\sim C'$ and $A\sim A'$ then there is a Koszul twisting morphism $\kappa'\colon C' \rightarrow A'$.
\end{proposition}

\subsection{Weight gradings and Koszul algebras over Koszul operads}
A \emph{weight grading} on an $\Op$-algebra $A$ is a decomposition
$$A = A(1) \oplus A(2) \oplus \ldots $$
which is compatible with the $\Op$-algebra structure in the sense that the structure map $\mu\colon \Op[A]\rightarrow A$ is weight preserving. The bar construction $B_\tau A = (\Coop[A],d+b)$ is then bigraded by weight $\ell_w$ and bar length $\ell_b$. Since $A$ is concentrated in positive weight, the bar construction is concentrated in the region $\{\ell_w\geq \ell_b\}$. Let $\Diag_b = \{\ell_w = \ell_b\}$ denote the diagonal. The coderivation $b$ preserves weight and decreases the bar length filtration. Let
$$A^\as := \Diag_b\cap \ker(b) \subseteq B_\tau A.$$
$A^\as$ is naturally a $\Coop$-coalgebra and the inclusion map $A^\as \rightarrow B_\tau A$ is a morphism of $\Coop$-coalgebras.

\begin{definition}
We say that the weight grading on $A$ is a \emph{Koszul weight grading} if the inclusion
$$A^\as\rightarrow B_\tau A$$
is a quasi-isomorphism. A \emph{Koszul $\Op$-algebra} is an $\Op$-algebra $A$ that admits a Koszul weight grading. If $A$ is Koszul then we call $A^\as$ the \emph{Koszul dual $\Coop$-coalgebra}.
\end{definition}

Let $C$ be a $\Coop$-coalgebra. A \emph{weight grading} on $C$ is a decomposition
$$C = C(1)\oplus C(2) \oplus \ldots$$
which is compatible with the coalgebra structure in the sense that $\Delta\colon C\rightarrow \Coop[C]$ preserves weight. The cobar construction $\Omega_\tau C = (\Op[C],d+\delta)$ is then bigraded by weight $\ell_w$ and cobar length $\ell_c$. Since $C$ is concentrated in positive weight, the cobar construction is concentrated in the region $\{\ell_w\geq \ell_c\}$ and we let $\Diag_c  = \{\ell_w = \ell_c\}$ denote the diagonal. The derivation $\delta$ preserves weight and increases the cobar length filtration. We define
$$C^\as := \Diag_c / \Diag_c\cap \im(\delta).$$
Then $C^\as$ is an $\Op$-algebra and the projection
$$f\colon \Omega_\tau C\rightarrow C^\as$$
is a morphism of $\Op$-algebras.

\begin{definition}
We say that a weight grading on $C$ is a \emph{Koszul weight grading} if the projection map $f\colon \Omega_\tau C\rightarrow C^\as$ is a quasi-isomorphism. A \emph{Koszul $\Coop$-coalgebra} is a $\Coop$-coalgebra $C$ that admits a Koszul weight grading. If $C$ is Koszul then we call $C^\as$ the \emph{Koszul dual $\Op$-algebra}.
\end{definition}

We will call a twisting morphism $\tau\colon \Coop \rightarrow \Op$ \emph{binary} if it vanishes outside arity $2$.

\begin{theorem} \label{thm:koszul}
Let $\kappa\colon C\rightarrow A$ be a Koszul twisting morphism relative to a binary Koszul twisting morphism $\tau\colon \Coop\rightarrow \Op$. If $C$ and $A$ have trivial differentials, then there are Koszul weight gradings on $C$ and $A$ such that $\phi_\kappa\colon C\rightarrow B_\tau A$ maps $C$ isomorphically onto $A^\as$ and $\psi_\kappa\colon\Omega_\tau C\rightarrow A$ factors through an isomorphism $C^\as\rightarrow A$.
\end{theorem}

\begin{proof}
Since $A$ has trivial differential, the quasi-isomorphism $\psi_\kappa\colon \Omega_\tau C \rightarrow A$ is surjective and induces an isomorphism $\HH_*(\Omega_\tau C) \cong A$. Since $C$ has trivial differential and since $\tau$ is binary, the derivation $\delta$ on $\Omega_\tau C$ increases cobar length by exactly $1$. Therefore we can introduce a weight grading on the homology $H = \HH_*(\Omega_\tau C)$ by
$$H(p) = \frac{\ker(\delta\colon \Op(p)\tensor_{\Sigma_p} C^{\tensor p} \rightarrow \Op(p+1)\tensor_{\Sigma_{p+1}} C^{\tensor p+1})}{\im(\delta\colon \Op(p-1)\tensor_{\Sigma_{p-1}} C^{\tensor p-1} \rightarrow \Op(p)\tensor_{\Sigma_p} C^{\tensor p})}$$
which we transport to a weight grading on $A$ via the isomorphism $\HH_*(\Omega_\tau C) \cong A$. It follows that $A(p)$ is the image of $\Op(p)\tensor_{\Sigma_p} C^{\tensor p}$ under $\psi_\kappa$.

Similarly, since $C$ has trivial differential, the quasi-isomorphism $\phi_\kappa\colon C\rightarrow B_\tau A$ is injective and induces an isomorphism $C\cong \HH_*(B_\tau A)$. Since $A$ has trivial differential and since $\tau$ is binary, the coderivation $b$ on the bar construction decreases bar length by exactly $1$ and as before we can introduce a weight grading on the homology $\HH_*(B_\tau A)$ which we transport to $C$. It follows that the induced weight grading on $C$ is given by $C(p) = \phi_\kappa^{-1}(\Coop(p)\tensor_{\Sigma_p} A^{\tensor p})$.

We will now verify that the weight gradings thus constructed are Koszul by proving that the image of $\phi_\kappa$ is equal to $A^\as = \Diag_b\cap\ker(b)$ and that $\psi_\kappa$ factors through an isomorphism $C^\as = \Diag_c/\Diag_c\cap \im(\delta)\cong A$. To this end, first note that $\kappa \colon C\rightarrow A$ vanishes on elements of weight different from $1$ and has image concentrated in weight $1$. This follows from the factorization of $\kappa$ as in the diagram \eqref{diag:universal}. Indeed, the weight grading on $C$ is inherited from the bar length grading on $B_\tau A$ and $\pi\colon B_\tau A\rightarrow A$ vanishes outside bar length $1$. Similarly, the image of $\iota\colon C\rightarrow \Omega_\tau C$ is contained in the cobar length $1$ component, whose image under $\psi_\kappa\colon \Omega_\tau C\rightarrow A$ is contained in $A(1)$. Recall that $\phi_\kappa$ is the composite
$$\xymatrix{C\ar[r]^-\Delta & \Coop[C] \ar[r]^-{\Coop[\kappa]} & \Coop[A].}$$
Since the image of $\kappa$ is contained in weight $1$, this shows that $\im(\phi_\kappa)\subseteq \Diag_b$. Since $C$ has trivial differential, $\im(\phi_\kappa)\subseteq \ker(b)$. Thus, $\im(\phi_\kappa)\subseteq \Diag_b\cap \ker(b)$. That we have equality follows because $\phi_\kappa$ is a quasi-isomorphism. Indeed, if $x\in \Diag_b\cap \ker(b)$ then the homology class of $x$ must be the image of the homology class of some $c\in C$, that is to say $\phi_\kappa(c)-x\in \im(b)$. But we also have that $\phi_\kappa(c)-x\in\Diag_b$. Since $\Diag_b\cap \im(b) = 0$ (which follows from the fact that $\ell_w\geq \ell_b$) we get $\phi_\kappa(c) = x$. Thus, the quasi-isomorphism $\phi_\kappa\colon C\rightarrow B_\tau A$ maps isomorphically onto $A^\as$, so indeed the weight grading on $A$ is Koszul.

Recall that $\psi_\kappa$ is the composite
$$\xymatrix{\Op[C] \ar[r]^-{\Op[\kappa]} & \Op[A] \ar[r]^-\mu & A.}$$
Since $\kappa$ vanishes outside weight $1$, this shows that $\psi_\kappa$ vanishes outside the diagonal $\Diag_c$. Moreover, since $A$ has trivial differential $\psi_\kappa$ factors through $C^\as =  \Diag_c/\Diag_c\cap\im(\delta)\rightarrow A$. Since $\psi_\kappa$ is surjective this map is necessarily surjective. That it is injective follows because $\psi_\kappa$ is a quasi-isomorphism. Indeed, if $x\in \Diag_c\cap\ker(\psi_\kappa)$ then $\delta(x) = 0$ since $\delta(\Diag_c) = 0$ as $\ell_w\geq \ell_c$ and $\delta$ preserves weight but increases cobar length. Thus, the homology class of $x$ gets mapped to zero under $\psi_\kappa$ which implies that $x\in \im(\delta)$ as $\psi_\kappa$ is a quasi-isomorphism. Thus, the quasi-isomorphism $\psi_\kappa\colon \Omega_\tau C \rightarrow A$ factors through an isomorphism $C^\as \cong A$ and therefore the weight grading on $C$ is Koszul.
\end{proof}

\begin{theorem} \label{thm:koszul algebra}
Let $\kappa\colon C\rightarrow A$ be a Koszul twisting morphism relative to a binary Koszul twisting morphism $\tau\colon \Coop\rightarrow \Op$ where $\Op$ is a Koszul operad. The following are equivalent:
\begin{enumerate}
\item $C$ and $A$ are formal. \label{1}
\item $A$ is formal and $\HH_*(A)$ is a Koszul $\Op$-algebra. \label{2}
\item $C$ is formal and $\HH_*(C)$ is a Koszul $\Coop$-coalgebra. \label{3}
\end{enumerate}
\end{theorem}

\begin{proof}
\eqref{1}$\Rightarrow$\eqref{2},\eqref{3}: If $C\sim \HH_*(C)$ and $A\sim \HH_*(A)$ then there is a Koszul twisting morphism $\kappa'\colon \HH_*(C)\rightarrow \HH_*(A)$ by Proposition \ref{prop:transfer}. By Theorem \ref{thm:koszul} this implies that $\HH_*(C)$ and $\HH_*(A)$ are Koszul and Koszul dual to each other.

\eqref{2}$\Rightarrow$\eqref{1}: We need to prove that $C$ is formal. Since $A$ is formal we may assume that $A = \HH_*(A)$ by Proposition \ref{prop:transfer}. Then we have a weak equivalence of $\Coop$-coalgebras $C\rightarrow B_\tau A$. Since $A=\HH_*(A)$ is assumed to be Koszul there is a weak equivalence $\HH_*(B_\tau A)\rightarrow B_\tau A$. So we get a zig-zag of weak equivalences $C\rightarrow B_\tau A \leftarrow \HH_*(B_\tau A)$ showing that $C$ is formal.

The proof of the implication \eqref{3}$\Rightarrow$\eqref{1} is similar.
\end{proof}

\begin{corollary} \label{cor:formal}
Let $\Op$ be a binary Koszul operad with Koszul dual cooperad $\Coop$ and let $\tau\colon \Coop\rightarrow \Op$ be the associated Koszul twisting morphism.
\begin{enumerate}
\item An $\Op$-algebra $A$ with zero differential is Koszul if and only if the bar construction $B_\tau A$ is formal as a $\Coop$-coalgebra.

\item A $\Coop$-coalgebra $C$ with zero differential is Koszul if and only if the cobar construction $\Omega_\tau C$ is formal as an $\Op$-algebra.
\end{enumerate}
\end{corollary}

If $\Op$ is a Koszul operad with Koszul dual cooperad $\Coop$, then the \emph{Koszul dual operad} is defined by $\Op^! = (\Sigma \Coop)^\vee$. If $A$ is a Koszul $\Op$-algebra with Koszul dual $\Coop$-coalgebra $C$, then the \emph{Koszul dual $\Op^!$-algebra} is defined to be the weight graded algebra $\Op^!$-algebra $A^!$ with $A^!(p) = (sC(p))^\vee$.

\begin{theorem} \label{thm:quadratic}
Let $A$ be a Koszul algebra over a binary Koszul operad $\Op$. Then $A$ is quadratic, that is, $A$ has a presentation of the form
$$A = \Op[V]/(R)$$
where $R\subseteq \Op[V](2) = \Op(2)\tensor_{\Sigma_2} V^{\tensor 2}$. Furthermore, the Koszul dual $\Op^!$-algebra $A^!$ is also Koszul and has quadratic presentation
$$A^! = \Op^![(sV)^\vee]/(R^\perp),$$
where $R^\perp\subseteq \Op^![(sV)^\vee](2)$ is the annihilator of $R\subseteq \Op[V](2)$ with respect to the induced pairing of degree $2$
$$\langle \,\, , \, \rangle \colon \Op^![(sV)^\vee](2) \tensor \Op[V](2) \rightarrow \QQ.$$
\end{theorem}

\begin{proof}
Theorem \ref{thm:koszul} identifies $A$ with $C^\as = \Diag_c/\Diag_c\cap \im(\delta)$ where $C= A^\as$. The diagonal $\Diag_c\subseteq \Omega_\tau C$ may be identified with the free $\Op$-algebra $\Op[V]$ where $V=C(1)$ is the weight $1$ component of $C$. We need to identify the generators of the $\Op$-ideal $\Diag_c\cap \im(\delta)\subseteq \Diag_c = \Op[V]$. By definition, the map $\delta$ is the unique $\Op$-derivation that makes the diagram
$$\xymatrix{C \ar[d]^-{\eta[C]} \ar[r]^-{\Delta} & \Coop[C] \ar[d]^-{\tau[C]} \\ \Op[C] \ar[r]^-\delta & \Op[C]}$$
commute. This implies that the $\Op$-ideal $\Diag_c\cap \im(\delta)$ is generated by the image of the map
$$\xymatrix{C \ar[r]^-\Delta & \Coop[C] \ar[r]^-{\tau[C]} & \Op[C]}$$
intersected with the diagonal. But if $\Op$ is binary, the twisting morphism $\tau$ vanishes outside arity $2$. Therefore, this image is contained in the subspace $\Op(2)\tensor_{\Sigma_2} V^{\tensor 2}$. This proves the first part of the theorem. The fact that $A^!$ is Koszul follows because its bar construction can be identified, up to a shift, with the weight graded dual of the cobar construction on $C$.

To prove the statement about the orthogonal presentation for $A^!$, first note that $\kappa\colon C\rightarrow A$ restricts to an isomorphism of weight $1$ components; $\kappa(1)\colon C(1) \cong A(1)$. We identify these components and write $V$ for both. By definition, the surjective morphism of $\Op$-algebras $\psi_\kappa \colon \Op[C]\rightarrow A$ is the composite
$$\xymatrix{\Op[C] \ar[r]^-{\Op[\kappa]} & \Op[A] \ar[r]^\mu & A}$$
It vanishes outside the diagonal $\Diag_c = \Op[V]$ and we use $\mu$ to denote also the induced surjective morphism of $\Op$-algebras $\Op[V]\rightarrow A$. Similarly, by definition the injective morphism of $\Coop$-coalgebras $\phi_\kappa \colon C\rightarrow \Coop[A]$ is the composite
$$\xymatrix{C \ar[r]^-\Delta & \Coop[C] \ar[r]^-{\Coop[\kappa]} & \Coop[A]}$$
It has image $\Diag_b\cap \ker b$ and in particular it factors through an injective morphism of $\Coop$-coalgebras $C\rightarrow \Diag_b = \Coop[V]$ that we will also denote by $\Delta$. By inspecting the definitions of the bar and cobar differentials $b$ and $\delta$ we have the following commutative diagram with exact rows in weight $2$:

\begin{equation} \label{ortho}
\xymatrix{0 \ar[r] & C(2) \ar@{=}[d] \ar[r]^-{\Delta(2)} & \Coop(2)\tensor_{\Sigma_2} V^{\tensor 2} \ar[d]_-\cong^-{\tau(2)\tensor 1^{\tensor 2}} \ar[r]^-{b(2)} & sA(2) \ar@{=}[d] \ar[r] & 0 \\
0 \ar[r] & C(2) \ar[r]^-{\delta(2)} & s\Op(2)\tensor_{\Sigma_2} V^{\tensor 2} \ar[r]^-{\mu(2)} & sA(2) \ar[r] & 0} 
\end{equation}

The injective morphism of weight graded $\Coop$-coalgebras $\Delta\colon C\rightarrow \Coop[V]$ dualizes (weight-wise) to a surjective morphism of $\Op^!$-algebras
$$\xymatrix{(s\Coop[V])^\vee \ar@{>>}[r]^-{(s\Delta)^\vee} \ar@{=}[d] & (sC)^\vee \ar@{=}[d] \\ \Op^![(sV)^\vee] \ar@{>>}[r] & A^!}$$
Since $A^!$ is Koszul the kernel of this morphism is generated in weight $2$ as an $\Op^!$-ideal. In weight $2$, diagram \eqref{ortho} shows that the kernel of $(s\Delta(2))^\vee$ may be identified with $R^\perp$, since the kernel of $\mu(2)$ is $R$ by definition.
\end{proof}

Let us be more explicit about the relation between $R$ and $R^\perp$. Let $W = (sV)^\vee$. The pairing $\langle \, , \, \rangle \colon W\tensor V\rightarrow \QQ$ of degree $1$ extends to a pairing of degree $2$
$$\langle \, , \, \rangle \colon \left( \Op^!(2)\tensor_{\Sigma_2} W^{\tensor 2} \right) \tensor \left( \Op(2)\tensor_{\Sigma_2} V^{\tensor 2} \right) \rightarrow \QQ$$
defined by
$$\langle (f;\alpha,\beta) , (\mu; x,y) \rangle = (-1)^\epsilon \langle f,\mu\rangle \langle \alpha , x\rangle \langle \beta, y\rangle + (-1)^\eta \langle f\tau,\mu\rangle \langle \beta,x\rangle \langle \alpha,y\rangle$$
where the signs are given by
$$\epsilon = |\mu|(|\alpha| + |\beta|) + |\beta||x| + |\alpha| + |x|$$
$$\eta = |\mu|(|\alpha|+|\beta|) + |\alpha||\beta| + |\alpha||x| + |\beta| + |x|$$

Every Koszul algebra over a binary Koszul operad is quadratic, but it is well known that not every quadratic algebra is Koszul. However, given a quadratic $\Op$-algebra $A$, the orthogonal presentation defines a quadratic $\Op^!$-algebra $A^!$ and there is an associated twisting morphism $\kappa \colon A^\as = (sA^!)^\vee \rightarrow A$. There is a natural chain complex associated to $\kappa$ called the \emph{Koszul complex}, which is acyclic if and only if $\kappa$ is a Koszul twisting morphism. This is a useful technique to prove that a given quadratic algebra is Koszul. We refer to \cite{Milles} for details.

\section{Rational homotopy theory and Koszul duality} \label{sec:rht}
The rational homotopy type of a simply connected space $X$ with finite Betti numbers is modeled in Sullivan's approach \cite{Sullivan} by a commutative differential graded algebra $A_{PL}^*(X)$ with cohomology $\HH^*(X;\QQ)$ and in Quillen's approach \cite{Quillen} by a differential graded Lie algebra $\lambda(X)$ with homology $\pi_*(\Omega X)\tensor \QQ$, see also \cite{FHT-RHT}.

There is a Koszul twisting morphism $\tau\colon \Com^\vee \rightarrow \Sigma \Lie$ where $\Com^\vee$ is the cooperad whose coalgebras are non-counital cocommutative coalgebras and $\Sigma \Lie$ is the operad whose algebras are suspensions of Lie algebras. The associated bar and cobar constructions are related to the classical Quillen functors \cite{Quillen,Neisendorfer,FHT-RHT}
$$\adjunction{DGC}{DGL}{\mathscr{C}}{\mathscr{L}}$$
as follows. For a dgc $C$ and a dgl $L$
\begin{equation*}
\Omega_\tau(\overline{C}) = s\mathscr{L}(C), \quad B_\tau(sL) = \QQ\oplus \mathscr{C}(L).
\end{equation*}

The \emph{Baues-Lemaire conjecture} \cite[Conjecture 3.5]{BL}, proved by Majewski \cite{Majewski} (see also \cite[Theorem 26.5]{FHT-RHT}), can be expressed in the language of Koszul duality theory in the following way:

\begin{quotation}
{\emph{Quillen's and Sullivan's approaches to rational homotopy theory are Koszul dual to one another under the Koszul duality between the commutative and the Lie operad.}}
\end{quotation}

A more precise statement is the following.
\begin{theorem} \label{thm:bl}
Let $X$ be a simply connected space with finite Betti numbers and let $M_X$ be its minimal model. There is a quasi-isomorphism of differential graded Lie algebras
$$\mathscr{L}(M_X^\vee)\stackrel{\sim}{\rightarrow} \lambda(X).$$
\end{theorem}

Whereas Sullivan's approach extends to nilpotent spaces, Quillen only defined his Lie model $\lambda(X)$ for simply connected spaces $X$. However, if $X$ is nilpotent and of finite $\QQ$-type, then its minimal model $M_X$ is of finite type and can therefore be dualized to a differential graded coalgebra $C_X=M_X^\vee$. The Lie algebra $\mathscr{L}(C_X)$ may serve as a replacement for the Quillen model, as justified by Theorem \ref{thm:bl}. This approach was carried out in \cite{Neisendorfer}.

\begin{theorem} [See {\cite[Proposition 8.2, 8.3]{Neisendorfer}}] \label{thm:nil}
Let $X$ be a nilpotent space of finite $\QQ$-type and let $M_X$ be its minimal model. Then there is an isomorphism of graded Lie algebras
$$\HH_{\geq 1} \mathscr{L}(M_X^\vee) \cong  \pi_{\geq 1}(\Omega X)\tensor \QQ.$$
Furthermore, there is an isomorphism of Lie algebras
$$\HH_0 \mathscr{L}(M_X^\vee) \cong l(\pi_1(X)\tensor \QQ),$$
where the right hand side is the Lie algebra of the Malcev completion of the fundamental group of $X$ \cite{Quillen}.
\end{theorem}

\begin{remark}
As stated, Theorem \ref{thm:nil} says nothing about the interpretation of the Lie bracket
$$\HH_0(\mathscr{L}(M_X^\vee)) \tensor \HH_n(\mathscr{L}(M_X^\vee)) \rightarrow \HH_n(\mathscr{L}(M_X^\vee))$$
for $n\geq 1$, but it is natural to expect that it corresponds to the action of $\pi_1(X)$ on the higher homotopy groups.
\end{remark}

\begin{proof}[Proof of Theorem \ref{thm:koszul space} and Theorem \ref{thm:koszul dual}]
Theorem \ref{thm:koszul dual} and the equivalence of \eqref{ks1}, \eqref{ks2} and \eqref{ks3} follow directly by combining Theorem \ref{thm:koszul algebra}, Theorem \ref{thm:bl} and Theorem \ref{thm:nil}. The equivalence \eqref{ks2}$\Leftrightarrow$\eqref{ks4} follows from the Sullivan-de Rham Equivalence Theorem \cite[Theorem 9.4]{BG} and the Localization Theorem \cite[Theorem 11.2]{BG}: If $X$ is connected, nilpotent and of finite $\QQ$-type, then the map $X\rightarrow \langle M_X \rangle$ is a rational homotopy equivalence, where $M_X$ is the minimal model for $X$. If $X$ is formal and $\HH^*(X;\QQ)$ is Koszul, then the minimal model for $X$ is also a minimal model for the cohomology and hence $\langle M_X\rangle$ represents the derived spatial realization of the Koszul algebra $\HH^*(X;\QQ)$. Conversely, if $X$ is rationally homotopy equivalent to the derived spatial realization of a Koszul algebra $A$, this means that $X\sim_\QQ \langle M_A \rangle$ where $M_A$ is the minimal model for the Koszul algebra $A$. By \cite[Theorem 9.4]{BG}, this implies that $M_A$ is also the minimal model for $X$, and since $M_A\stackrel{\sim}{\rightarrow} A$ this implies that $X$ is formal and that $\HH^*(X;\QQ) \cong A$.
\end{proof}

\section{Koszul duality for Gerstenhaber $n$-algebras and rational homology of $n$-fold loop spaces} \label{sec:g}
A \emph{Gerstenhaber $n$-algebra} is a graded vector space $A$ together with two binary operations
$$\mu \colon A_p\tensor A_q\rightarrow A_{p+q}, \quad \lambda \colon A_p\tensor A_q\rightarrow A_{p+q+n-1},$$
such that for all $x,y,z\in A$, writing $xy = \mu(x\tensor y)$ and $[x,y] = \lambda(x\tensor y)$,
$$xy = (-1)^{|x||y|}yx, \quad (xy)z = x(yz),$$
$$[x,y] = -(-1)^{(|x|+n-1)(|y|+n-1)}[y,x],$$
$$[x,[y,z]] = [[x,y],z] + (-1)^{(|x|+n-1)(|y|+n-1)}[y,[x,z]],$$
$$[x,yz] = [x,y]z + (-1)^{(|x|+n-1)|y|} y[x,z].$$

\begin{theorem}[F.Cohen \cite{fcohen}]
The homology of the little $n$-cubes operad $E_n$ is isomorphic to the operad $\G_n$ of Gerstenhaber $n$-algebras. In particular, the homology of any $E_n$-algebra has the structure of a $\G_n$-algebra.
\end{theorem}
The prototypical example of an $E_n$-algebra is the $n$-fold loop space $\Omega^n X$ of a based topological space $X$ \cite{BV,May}. The homology $\HH_*(\Omega^n X;\QQ)$ is a $\G_n$-algebra where $\mu$ is the Pontryagin product and $\lambda$ the Browder bracket. Any graded commutative algebra $A$ may be viewed as a $\G_n$-algebra by setting $\lambda = 0$.

\begin{theorem}[Getzler-Jones \cite{GJ}]
The operad $\G_n$ is Koszul and it is Koszul self-dual up to a suspension; $\G_n^! = \Sigma^{n-1} \G_n$.
\end{theorem}

If $A$ is a Koszul $\G_n$-algebra, then the Koszul dual $A^!$ as defined in Section \ref{sec:koszul} is an algebra over $\G_n^! = \Sigma^{n-1} \G_n$. In order to get an algebra over $\G_n$ again, we define the \emph{Koszul dual $\G_n$-algebra of $A$} to be the desuspension $A^{!_{\G_n}} = s^{1-n} A^!$.

\begin{proposition} \label{prop:comgn}
Let $A$ be a commutative algebra viewed as a $\G_n$-algebra with trivial Lie bracket. Then $A$ is Koszul as a commutative algebra if and only if it is Koszul as a $\G_n$-algebra. In this situation there is an isomorphism of $\G_n$-algebras
$$A^{!_{\G_n}} \cong \Lambda (s^{1-n} A^{!_{\Lie}}).$$
\end{proposition}

\begin{proof}
This is a consequence of the decomposition of operads $\G_n = \Com \circ \Lie_{n-1}$, where $\Lie_{n-1} = \Sigma^{n-1}\Lie$, see \cite[13.3.14]{LV}. Since the Lie algebra structure on $A$ is trivial, the $\G_n$ bar construction of $A$ is isomorphic to
$$B_{\G_n}(A) \cong \Lie_{n-1}^\as[B_{\Com}(A)].$$
It follows that $B_{\G_n}(A)$ is formal as a $\G_n^\as$-coalgebra if and only if $B_{\Com}(A)$ is formal as a $\Com^\as$-coalgebra. This proves the first part of the proposition. For the second part, simply observe that $\Lie_{n-1}^! = \Sigma^{n-1}\Com$.
\end{proof}

\begin{proposition} \label{prop:mm}
For any $n$-connected space $X$ there is an isomorphism of $\G_n$-algebras
$$\HH_*(\Omega^n X;\QQ) \cong \Lambda(s^{-n} \pi_*(X)\tensor \QQ)$$
where the Lie bracket on the right hand side is induced by Whitehead products on the homotopy groups of $X$.
\end{proposition}

\begin{proof}
By the Milnor-Moore theorem \cite[p.263]{MM}, since $\Omega^n X$ is connected the Hurewicz homomorphism
$$h\colon \pi_*(\Omega^n X)\tensor\QQ \rightarrow \HH_*(\Omega^n X;\QQ)$$
induces an isomorphism of graded algebras between $\HH_*(\Omega^n X;\QQ)$ and the universal enveloping algebra of the graded Lie algebra $\pi_*(\Omega^n X)\tensor\QQ$ with the Samelson product. But for $n\geq 2$, the Samelson product on $\pi_*(\Omega^n X)$ is trivial, so the universal enveloping algebra is a free graded commutative algebra and we get that the Hurewicz homomorphism induces an isomorphism of graded algebras
$$\Lambda (\pi_*(\Omega^n X)\tensor \QQ) \stackrel{\cong}{\rightarrow} \HH_*(\Omega^n X;\QQ).$$
To finish the proof we note that under the identification $\pi_*(\Omega^n X) = s^{-n}\pi_*(X)$, the Hurewicz homomorphism takes Whitehead products in $\pi_*(X)$ to Browder brackets in $\HH_*(\Omega^n X;\QQ)$ \cite[p.215, p.318]{fcohen}.
\end{proof}

\begin{proof}[Proof of Theorem \ref{thm:loop}]
The proof consists in assembling three facts:
\begin{itemize}
\item For any Koszul space $X$ there is an isomorphism of graded Lie algebras
$$s^{-1} \pi_*(X)\tensor \QQ \cong \HH^*(X;\QQ)^{!_{\Lie}}$$
where the Lie bracket on the left hand side is induced by Whitehead products on $\pi_*(X)$ (Theorem \ref{thm:koszul dual}).

\item For any commutative Koszul algebra $A$ viewed as a $\G_n$-algebra with trivial Lie bracket there is an isomorphism of $\G_n$-algebras
$$A^{!_{\G_n}} \cong \Lambda(s^{1-n} A^{!_{\Lie}})$$
(Proposition \ref{prop:comgn}).

\item For any $n$-connected space $X$ there is an isomorphism of $\G_n$-algebras
$$\HH_*(\Omega^n X;\QQ) \cong \Lambda(s^{-n} \pi_*(X)\tensor \QQ)$$
where the Lie bracket on the right hand side is induced by Whitehead products on $\pi_*(X)$ (Proposition \ref{prop:mm}).
\end{itemize}
\end{proof}

\section{Examples of Koszul spaces} \label{sec:examples}
\begin{example} \emph{Spheres}. The first example of a Koszul space is the sphere $S^n$, where $n\geq 1$. Spheres are formal, and the cohomology algebra $\HH^*(S^n;\QQ)$ is generated by a class $x$ of cohomological degree $n$ modulo the relation $x^2 = 0$. The Koszul dual Lie algebra is the free graded Lie algebra $\LL(\alpha)$ on a generator $\alpha$ of homological degree $n-1$.
$$\pi_*(\Omega S^n)\tensor \QQ \cong \LL(\alpha) = \left\{ \begin{array}{cc} \langle \alpha, [\alpha,\alpha] \rangle_\QQ, & \mbox{$n$ even} \\ \langle \alpha \rangle_\QQ, & \mbox{$n$ odd}  \end{array} \right.$$
This gives yet another way of seeing Serre's \cite{Serre-h} classical result that $\pi_i(S^n)\tensor \QQ = \QQ$ for $i = n$, and $i = 2n-1$ if $n$ is even, and $\pi_i(S^n)\tensor \QQ = 0$ else.
\end{example}

\begin{example} \emph{Suspensions}. More generally, the suspension $\Sigma X$ of any connected space $X$ is a Koszul space. A suspension is formal (indeed rationally homotopy equivalent to a wedge of spheres, see \cite[Theorem 24.5]{FHT-RHT}) and the cohomology $\HH^*(\Sigma X;\QQ)$ is the trivial algebra $\perp(\rH{}^*(\Sigma X;\QQ))$ generated by the reduced cohomology. Here $\perp(V) = \Lambda(V)/(\Lambda^2 V)$. The Koszul dual is a free graded Lie algebra on the reduced homology of $X$
$$\pi_*(\Omega \Sigma X)\tensor \QQ \cong \LL(\rH_*(X;\QQ)).$$
\end{example}

\begin{example} \emph{Loop spaces}. The loop space $\Omega X$ of any $1$-connected space $X$ is a Koszul space with cohomology the free graded commutative algebra $\Lambda(V)$ on a graded vector space $V$. The Koszul dual Lie algebra is the abelian Lie algebra
$$\pi_*(\Omega^2 X)\tensor \QQ \cong s^{-2} V^\vee.$$
\end{example}

\begin{example} \emph{Products and wedges}.
If $X$ and $Y$ are Koszul spaces then so are $X\times Y$ and $X\vee Y$. Indeed, it is easy to check that products and wedges of (co)formal spaces are (co)formal, see e.g., \cite[Lemma 4.1]{NM}. On the level of cohomology, this is reflected by the fact that tensor and fiber products of Koszul algebras are Koszul:
\begin{align*}
\HH^*(X\times Y;\QQ) & \cong \HH^*(X;\QQ)\tensor \HH^*(Y;\QQ), \\
\HH^*(X\vee Y;\QQ) & \cong \HH^*(X;\QQ)\times_\QQ \HH^*(Y;\QQ).
\end{align*}
\end{example}

\begin{example} \emph{Configuration spaces.}
For any $k$ and $n$ the configuration space $F(\RR^n,k)$ of $k$ points in $\RR^n$ is a Koszul space. This follows because configuration spaces are formal (see e.g. \cite{LamV}) and their cohomology algebras are Koszul. The rational cohomology algebra is generated by elements $a_{pq}$ of cohomological degree $n-1$ for $1\leq p<q\leq k$ subject to the `Arnold relations'
\begin{eqnarray*}
& a_{pq}a_{qr} + a_{qr}a_{rp} + a_{rp}a_{pq} = 0,\quad \mbox{$p,q,r$ distinct}, \\
& a_{pq}^2 = 0, \quad \mbox{($n$ odd)}
\end{eqnarray*}
Here we use the convention $a_{pq} = (-1)^n a_{qp}$ for $p>q$. This algebra has a PBW-basis \cite{Priddy} consisting of all monomials $a_{i_1 j_1}\ldots a_{i_r j_r}$ where $i_1<\ldots < i_r$ and $i_p<j_p$ for all $p$, and therefore it is Koszul. Hence, we have an isomorphism of graded Lie algebras $\pi_*(\Omega F(\RR^n,k))\tensor \QQ = \HH^*(F(\RR^n,k);\QQ)^{!_{\Lie}}$. When calculating the orthogonal relations we recover \cite[Theorem 2.3]{CG}: as a graded Lie algebra $\pi_*(\Omega F(\RR^n,k))\tensor \QQ$ is generated by classes $\alpha_{pq}$ of homological degree $n-2$ for $1\leq p<q\leq k$ subject to the `infinitesimal braid relations' or `Yang-Baxter Lie algebra relations'
\begin{align*}
[\alpha_{pq},\alpha_{rs}] & = 0, \quad \{p,q\}\cap \{r,s\} = \emptyset, \\
[\alpha_{pq}, \alpha_{pr} + \alpha_{qr}] & = 0, \quad \mbox{$p,q,r$ distinct}.
\end{align*}
Again, we use the convention that $\alpha_{qp} = (-1)^n \alpha_{qp}$ for $p>q$. These presentations for the cohomology and homotopy Lie algebra of configuration spaces are well-known, but we stress that it is possible to \emph{derive} the presentation of $\pi_*(\Omega F(\RR^n,k))\tensor \QQ$ from the presentation of $\HH^*(F(\RR^n,k);\QQ)$ using the fact that $F(\RR^n,k)$ is a Koszul space.
\end{example}

\begin{example} \emph{Highly connected manifolds.} \label{cor:dim}
Let $M$ be a $(d-1)$-connected closed $m$-dimensional manifold with $m\leq 3d-2$, $d\geq 2$ and where $\dim_\QQ \HH^*(M;\QQ) \geq 4$. By \cite[Proposition 4.4]{NM} any such manifold is formal and coformal, whence a Koszul space by Theorem \ref{thm:koszul space}. By Poincar\'e duality and for degree reasons, the rational cohomology of $M$ admits a basis $1,x_1,\ldots,x_n,\omega$ where $\omega$ is a generator of $\HH^m(M;\QQ) \cong \QQ$ and where $x_1,\ldots,x_n$ are indecomposable with respect to the cup product. Therefore, the cohomology algebra is completely determined by the structure coefficients $q_{ij}\in\QQ$ where
$$x_i x_j = q_{ij} \omega.$$
By Theorem \ref{thm:koszul dual}, we can compute the rational homotopy Lie algebra of $M$ by finding the orthogonal relations. Clearly, a relation
$$\sum_{i,j} c_{ij} x_ix_j = 0$$
holds if and only if $\sum_{i,j} q_{ij} c_{ij} = 0$. So we get that $\pi_*(\Omega M)\tensor \QQ$ is a free graded Lie algebra on classes $\alpha_1,\ldots,\alpha_n$ modulo a single quadratic Lie form
$$\pi_*(\Omega M)\tensor \QQ \cong \LL(\alpha_1,\ldots,\alpha_n)/(Q),$$
$$Q = \sum_{i,j} (-1)^{|x_i||\alpha_j|}q_{ij} [\alpha_i,\alpha_j].$$
This agrees with the presentation obtained in \cite[Theorem 5]{Neisendorfer-ci} up to a sign.
\end{example}

\end{document}